\documentclass[a4paper]{article}
\usepackage {CJKutf8}
\usepackage{amsmath}
\usepackage{amssymb}
\usepackage{amsthm}
\usepackage{amsrefs}
\usepackage{hyperref}

\theoremstyle{plain}%
\newtheorem{theorem}{Theorem}
\newtheorem{proposition}{Proposition}%
\newtheorem{lemma}{Lemma}
\newtheorem{corollary}{Corollary}

\theoremstyle{remark}%
\newtheorem{remark}{Remark}%

\theoremstyle{definition}%
\newtheorem{definition}{Definition}%

\numberwithin{equation}{section}

\newcommand{\inn}[2]{\left\langle#1,\,#2\right\rangle}
\DeclareMathOperator{\ran}{ran}
\DeclareMathOperator{\spn}{span}

\newcommand{\grad}{\nabla}
\newcommand{\Lap}{\Delta}
\newcommand{\di}{\partial}

\DeclareMathOperator{\Ric}{Ric}
\DeclareMathOperator{\Sc}{Sc}

\newcommand{\eps}{\epsilon}
\newcommand{\ls}{\lesssim}
\newcommand{\g}{\gamma}
\newcommand{\al}{\alpha}
\newcommand{\Si}{\Sigma}

\renewcommand{\b}{\bar}

\newcommand{\Rb}{\mathbb{R}}

\newcommand{\W}{\mathcal{W}}

\newcommand{\Ga}{\Gamma}

\renewcommand{\l}{\lambda} 
\newcommand{\Acirc}{\mathring{A}} 
\newcommand{\Sb}{\mathbb{S}} 
\renewcommand{\b}{\bar} 
\newcommand{\z}{\zeta}
\newcommand{\OO}[2][k]{O_{H^{#1}}(#2)}

\newcommand{\abs}[1]{\left\lvert#1\right\rvert}
\newcommand{\norm}[1]{\lVert#1\rVert}

\newcommand{\Set}[1]{\left\{#1\right\}}
\newcommand{\fullfunction}[5]{\ensuremath{
		\begin{array}{ccrcl}
			{#1}    & \colon  & {#2} & \longrightarrow & {#3} \\
			\mbox{} & \mbox{} & {#4} & \longmapsto     & {#5}
\end{array}}}

\newcommand{\od}[2]{\ensuremath{\frac{d#1}{d#2}}}
\newcommand{\md}[6]{\ensuremath{
		\ifinner
		\tfrac{\partial{^{#2}}#1}{\partial{#3^{#4}}\partial{#5^{#6}}}
		\else
		\tfrac{\partial{^{#2}}#1}{\partial{#3^{#4}}\partial{#5^{#6}}}
		\fi
}}
\newcommand{\del}[1]{\left(#1\right)}
\newcommand{\thmref}[1]{Theorem~\ref{#1}}

\newcommand{\defnref}[1]{Definition~\ref{#1}}
\newcommand{\secref}[1]{Section~\ref{#1}}
\newcommand{\lemref}[1]{Lemma~\ref{#1}}
\newcommand{\propref}[1]{Proposition~\ref{#1}}
\newcommand{\remref}[1]{Remark~\ref{#1}}

\begin{document}
	
	\title{Adiabatic theory for the area-constrained Willmore flow}
	\author{Jingxuan Zhang 
		(\begin{CJK*}{UTF8}{gbsn}张景宣
		\end{CJK*})$^1$
		\footnote{Email: \url{jingxuan.zhang@math.ku.dk}}
	}
	\date{%
		$^1$	University of Copenhagen\\
		Department of Mathematical Sciences\\
		Universitetsparken 5\\
		2100 Copenhagen, Denmark\\[2ex]%
		\today
	}
	\maketitle
	
	\begin{abstract}
		In this paper, we develop an adiabatic theory for the evolution of large closed surfaces under the area-constrained Willmore (ACW) flow in a three-dimensional asymptotically Schwarzschild manifold. We construct explicitly a map, defined on a certain four-dimensional manifold of barycenters, which characterizes key static and dynamical properties of the ACW flow. In particular, using this map, we find an explicit four-dimensional effective dynamics of barycenters, which serves as a uniform asymptotic approximation for the (infinite-dimensional) ACW flow, so long as the initial surface satisfies certain mild geometric constraints (which determine the validity interval). Conversely, given any prescribed flow of barycenters evolving according to this effective dynamics, we construct a family of surfaces evolving by the ACW flow, whose barycenters are uniformly close to the prescribed ones on a large time interval (whose size depends on the geometric constraints of initial configurations).
	\end{abstract}

	\section{Introduction}
	Let $(M,g)$ be a $3$-dimensional, complete, oriented Riemannian manifold with non-negative curvature. Consider the area-constrained Willmore (ACW) flow,
	\begin{equation}
		\label{W}
		\di_t x ^N = -W(x)-\l H(x).
	\end{equation}
	Here, for $t\ge0$, $x=x_t:\Sb \to M$ is a family of embeddings of spheres (with orientation compatible with that on $M$). $\di_tx^N$ denotes the normal velocity at $x$,
	given by $\di_tx^N:=g(\di_tx,\nu)$, where $\nu=\nu(x)$ is the unit normal vector to $\Si$ at $x$.
	$H(x)$ denotes the mean curvature scalar at $x$, and
	$$W(x):=\Lap H(x)+H(x)(\Ric_M(\nu,\nu)+\lvert\Acirc\rvert^2(x))$$ is the Willmore operator,
	where $\Acirc(x)$ denotes the traceless part of the second fundamental form.
	The number $\l=\l(t)$ is given by
	$$\l=\frac{1}{\int_\Si H^2}\int_\Si\del{\abs{\grad H}^2-H^2\Ric_M(\nu,\nu)-H^2\lvert\Acirc\rvert^2}\,d\mu,$$
	where $\mu=\mu_\Si^g$ is the  canonical measure on $\Si$, induced by the embedding $x$ and background metric $g$.
	This choice of $\l$ ensures that \eqref{W}
	is area-preserving, see Appendix \ref{sec:ap}.
	
	Mathematically, the ACW flow \eqref{W} can be viewed as the $L^2$-gradient flow
	of the Willmore energy 
	$$
	\W(\Si)=\frac{1}{4}\int_\Si H^2\,d\mu,
	$$
	which is a well-defined $C^2$-functional in a suitable configuration space. 
	Denote by $d\W(x)$ the Fr\'echet derivative of $\W$.
	Define the normal $L^2$-gradient $\grad ^N \W(x)\phi:=d\W(x)\phi$  for every normal,
	area-preserving variation 
	$\phi$ on the surface $\Si=x(\Sb)$. 
	Then
	$\grad ^N \W(x)$ is given by the r.h.s. of \eqref{W}.
	Hence, for a flow of immersions $x=x_t,\,t\ge0$ in the configuration space of $\W$, we can rewrite \eqref{W} as 
	$$\di_t x^N=\grad ^N \W(x).$$
	See \secref{sec:config} for more details.
	
	Physically, the ACW flow \eqref{W} is the gradient flow of the Hawking mass of embedded $2$-spheres,
	defined in \eqref{haw} below and  first introduced in 
	\cite{MR3960907}. In mathematical biology, building upon Helfrich's seminal work
	\cite{Helfrich}, the flow \eqref{W} 
	and its variants are used to model the motion of  cell membranes.
	More recently, \eqref{W} finds applications to nonlinear elasticity \cites{FGJ,MR1916989},
	among various other fields in applied science.
	
	\subsection{Main results}\label{sec:main}
	In this paper, we develop a rigorous adiabatic (or slow-motion approximation)
	theory for \eqref{W}. To state our main result, 
	denote by $H^k$ the Sobolev space of order $k$. 
	Let
	$$
	X^k:= H^k(\Sb,M),\quad X^k_c:=\Set{x\in X^k:\abs{x(\Sb)}=c}.
	$$
	Here and below, for a surface $\Si:=x(\Sb)\subset M$, denote by $$\abs{\Si}:=\int_\Si\,d\mu$$ the area of $\Si$.
	Following the nomenclature in \cite{MR2785762}, we call (the images of) static solutions to \eqref{W} \textit{surfaces of Willmore type}.

	Then we prove the following:
	\begin{theorem}[Main]
		\label{thm3}
		Let $k\ge4,\,c\gg1$. 
		Fix  $R\gg1,\,\delta\ll1$ in \defnref{defn1.1}
		of admissible surfaces.
		
		Then there exists 
		a map $$\Psi:M':=\Rb_{>R}\times B_1(0)\subset \Rb\times \Rb^3\to X^k$$ with 
		the following property: Denote by $\z\equiv(r,z)\in M'$
		and $v_\z\equiv \Psi(\z)\in X^k$. 
		\begin{enumerate}
			\item (Critical point) The immersion $v_\zeta$ parametrizes a surface of Willmore type if and only if $\z$
			is a critical point of the function $\W\circ \Psi:\Rb^4\to \Rb$, 
			restricted to the submanifold $\Set{\z\in M': \abs{v_\z(\Sb)}=c}$.
			\item (Stability) Suppose $v_\zeta$ parametrize an admissible surface of Willmore type.
			Then $v_\zeta$ is uniformly stable
			with small
			area-preserving $H^k$-perturbation		
			\footnote{This means that if $y_0$ is another admissible surface that is
				$H^k$-close to $v_\zeta$ in
				the topology given in \defnref{defn2}, then for every $\eps>0$ there
				exists $T>0$ such that  $\norm{v_t-y_t}_{X^k}<\eps$ for all $t\ge T$,
				where $v_t,\,y_t$ are respectively the flows generated by $v_\zeta,\,y_0$
				under \eqref{W}.} 
			if $\z$ is a strict local
			minimum of the function $\W\circ \Psi$
			restricted to the submanifold $\Set{\z\in M': \abs{v_\z(\Sb)}=c}$.
			\item (Effective dynamics) 	
			Let $\Si_*=x_*(\Sb)$ be an admissible surface. Let $\Si_t=x_t(\Sb),\,t\ge0$
			be the global solution to \eqref{W}
			with initial configuration $\Si|_{t=0}=\Si_*$ as in \thmref{thm1}.
			
			Then there exist $\al>0$,  $T=O(R^{-\al})$, and
			a path $\z_t\in M',\,t\ge0$, such that
			for every $t\ge T$, there holds
			\begin{equation}
				\label{2.0}
				\norm{v_{\z_t}-x_t}_{X^k}=O(R^{-3}).
			\end{equation}
			
			Moreover, the path $\z_t\equiv(r_t,z_t)$ evolves according to 
			\begin{align}
				\label{2}
				\dot z=&\frac{1}{4\pi}\grad_z \W\circ \Psi (r,z)+O(R^{-3}),\\
				\label{2.0.1}
				\dot r=&4 R^{-2}+O(R^{-3}).
			\end{align}
			In \eqref{2} the leading term is of the order $O(R^{-2})$.
			
			\item Conversely, if $\z_t\in M'$ is a flow evolving according to \eqref{2}-\eqref{2.0.1},
			then there exists a global solution $x_t$ to \eqref{W} such that \eqref{2.0} holds for this choice  of $\z_t$ and every
			$T\le t\le T+R$. 
		\end{enumerate}
	\end{theorem}

	\begin{remark}
		The map $\Psi$ in \thmref{thm3} is defined by $\Psi(r,z)=\theta(\Phi(r,z),r,z)$,
		where $\theta,\, \Phi$ are given respectively in \eqref{3} and \defnref{LSM}.
		Following \cite{MR3824945}, we call $\Phi$ (or equivalently $\Psi$) \textit{the Lyapunov-Schmidt map}.
	\end{remark}	
	\begin{remark}
		The content of \thmref{thm3} says that the static and dynamical properties of \eqref{1}
		are  captured by the effective action $\W\circ \Psi$, 
		on the  $3$-manifold given by the level set $\Set{\z\in M': \abs{v_\z(\Sb)}=c}$.
	\end{remark}
	\begin{remark}	
		\eqref{2.0}-\eqref{2.0.1} constitute the crucial part of the theorem,
		which says that the infinite dimensional dynamical system \eqref{W} reduces 
		uniformly in time to the finite system of ODEs, \eqref{2}-\eqref{2.0.1}.
	\end{remark}
	
	
	On the analytical ground, roughly speaking, \thmref{thm3} shows that the pullback by the Lyapunov-Schmidt map
	$\Psi$ to $M'$ essentially reduces key static and dynamical properties of
	the infinite dimensional dynamical system \eqref{W}
	to finite dimensional ones. 
	
	For applications to various physical models, our results in this paper provide an explicit
	$4$-dimensional dynamical system, \eqref{2},
	whose behavior approximates that of \eqref{W} uniformly for all sufficiently large time,
	and captures key qualitative behaviours of \eqref{W}. 
	This can potentially reduce computational complexity for the
	complicated forth order PDE \eqref{W}.

	\subsection{Historical remarks}
	The problem we study here is motivated by a recent work  \cite{eichmair2021large} by M. Eichmair and T. Koerber,
	in which the authors study stationary solutions to \eqref{W} using Lyapunov-Schmidt reduction.
	Here we derive some dynamical analogues of the static existence
	results in \cite{eichmair2021large}, with, however, rather different focus. Indeed,
	the main point of our results is that we have  1. \textit{explicit} information about the adiabatic
	parts of a flow of surfaces evolving according to \eqref{W}, with 2. \textit{uniformly small} errors in time,
	and 3. we can construct solution \eqref{W} with \textit{prescribe} adiabatic behavior. 
	See the precise statements of these points in \thmref{thm3}.
	
	The method of adiabatic approximation has a long history in classical field theory.
	See \cite{MR2360179} for an excellent review in this context. 
	Our work here is inspired by
	a series of papers by I.M. Sigal with several co-authors, in which the adiabatic theory
	is adapted to geometric problems \cites{sigal2012stability,MR2465296,MR3803553,MR3824945}.
	Other results along this line which we have referred to include \cites{MR1369419,MR3251832,MR4160305,chodosh2017global,MR4303943,MR3397388},
	which cover a range of static and dynamical problems of geometric equations using
	Lyapunov-Schmidt reduction.

	Among the papers above, we single out a recent paper \cite{MR3824945}, in which the authors study a formally similar problem (namely, the volume-preserving mean curvature flow with initial configurations close to small geodesics spheres), from which we draw much inspiration.
	In particular, it appears that the notion of Lyapunov-Schmidt map  
	is first mentioned in this paper.
	
	It seems to us that our results is the first rigorous adiabatic theory for the area-constrained  Willmore flow.
	We expect these results to be robust, in the sense that they can be easily adapted to problems related to \eqref{W}, for instance, using the generalized Willmore energy developed in \cite{Friedrich2020} (which covers, among others, the applications to biomembranes). 
	In a separate paper, we will treat the abstract properties of the Lyapunov-Schmidt map defined
	in Appendix \ref{A0}.

	\section{Setup of the problem}
	\subsection{Asymptotically Schwarzschild manifolds}
	\label{sec:1.1}
	A  $3$-dimensional complete Riemannian manifold $(M,g)$ is said to be
	$C^k$-close to Schwarzschild if the following holds:
	\begin{enumerate}
		\item $M\setminus K$ is diffeomorphic
		to $\Rb^3\setminus \overline{B_1(0)}$ for some compact subset $K\subset M$.
		\item The metric  $g$ splits as $g_S+h$,
		where $$g_S:=\del{1+\frac{m}{2\abs{x}}}^4\delta_{ij}$$ is the Schwarzschild metric
		with ADM mass $m>0$, 
		and $h\in C^k$ is
		a small perturbation satisfying 
		\begin{equation}\label{h}
			h_{ij}=h_{ji}, \quad\di^\alpha h_{ij}\le \eta \abs{x}^{-(2+\abs{\alpha})}\quad (\abs{\alpha}\le k,\;\abs{x}\gg1),
		\end{equation} 
		for some fixed
		small decay coefficient $\eta\ll1$.
		Here $x\in \Rb^3$ denotes the coordinate in the 
		asymptotic chart on $M$.
	\end{enumerate}
	Physically, for applications to  general relativity,
	such manifold $M$ is a perturbation of the static Schwarzschild
	black hole $(\Rb^3\setminus \overline{B_{m/2}(0)},g_S)$. 
	
	To simplify notations, throughout the paper we normalize ADM mass to be $m=2$.
	We assume that the ambient space $M$ is $C^k$-close to Schwarzschild for sufficiently large $k$,
	and that in \eqref{h} the decay coefficient $\eta\ll1$.
	Thus in what follows we take 
	$$(M,g)= (\Rb^3\setminus \overline{B_{1}(0)},g_S+h)$$ where $h$ is as in \eqref{h}. 
	
	To use results in \cites{MR2785762,MR4236532,eichmair2021large}, we assume the scalar curvature
	$\Sc$ on $M$ satisfies the following decay properties:
	\begin{align}
		\label{1.2}
		x^j\di_{x^j}\del{\abs{x}^2\Sc}=&o(\abs{x}^2),\\
		\label{1.3}
		\Sc(x)-\Sc(-x)=	&o(\abs{x}^4).
	\end{align}
	The asymptotic flatness
	condition \eqref{1.2} is satisfied if $g$ is $C^k$-close to Schwarzschild with $k\ge 4$,
	and  $\Sc=o(\abs{x}^4)$, in
	which case $\grad \Sc=o(\abs{x}^{-5})$. 
	\eqref{1.3} means the 
	scalar curvature on $M$ is asymptotically even.
	Geometrically, condition \eqref{1.2} provides quantitative control
	for various estimates involving extrinsic geometric quantities.
	Condition \eqref{1.3} provides qualitative control of the
	effective action in Secs. \ref{sec:3}, \ref{sec:A2}.
	
	\subsection{The geometric structure of \eqref{W}}
	\label{sec:config}
	In this subsection, we lay out the geometric structure of ACW flow \eqref{W}.
	This structure is understood in the subsequent developments in Secs. \ref{sec:3}-\ref{sec:4}.
	
	Let $c\gg1, k\ge4$ be given. 
	Recall that in \secref{sec:main}, we have defined the configuration spaces
	\begin{equation}
		\label{1}
		X^k:= H^k(\Sb,M),\quad X^k_c:=\Set{x\in X^k:\abs{x(\Sb)}=c},
	\end{equation}
	where $\abs{\Si}:=\int_\Si\,d\mu_\Si^g$ denotes the area of $\Si$ w.r.t. the embedding $x$ and background metric $g$.
	One can check easily that \eqref{W} is well-defined in $X^k_c$.
	The spaces in \eqref{1} are equipped with the $L^2$-inner product
	\begin{equation}
		\label{inn}
		\inn{\phi}{\phi'}:=\int_\Sb \inn{\phi}{\phi'}_{\text{Euclidean}}\quad (\phi,\phi'\in X^k).
	\end{equation}
	
	Let $x\in X^k$ and write $\Si=x(\Sb)$. The tangent spaces to $x$ at $X^k$ and $X^k_c$ are respectively  given by
	\begin{align}
		\label{tan1}
		T_x X^k =& X^k,\\
		\label{tan2}
		T_x X^k_c=&\Set{\phi\in T_x X^k:\int_\Si H  g(\phi,\nu)=0}.
	\end{align}
	Here, \eqref{tan2} is due to the well-known first variation formula of the area functional.
	Notice that, slightly abusing notation, in \eqref{tan2} we view $\phi$ as a 
	vector field over $\Si$. 
	With \eqref{inn}, we have a formal Riemannian structure on the
	configuration spaces $X^k$ and $X^k_c$.

	With this geometric structure of $X^k$, one can view
	the equation \eqref{W} as the $L^2$-gradient flow, restricted to $X^k_c$,
	of the Willmore energy
	\begin{equation}
		\label{WE}
		\W(\Si)=\frac{1}{4}\int_\Si H^2\,d\mu_\Si^g.
	\end{equation}
	Using Sobolev inequalities, one can show that for 
	$k\ge4$, the functional $\W$ is well-defined and 
	$C^2$ (in the sense of Fr\'echet derivatives) on $X^k_c$.
	
	Let $d\W(x):T_x X^k_c\to T_x X^{k-4}_c$ be the Fr\'echet derivative of $\W$ at an embedding
	$x$ in the class $X^k_c$. 
	Define the normal $L^2$-gradient $\grad ^N \W(x)\phi:=d\W(x)\phi$  for every normal,
	area-preserving variation 
	$\phi$ on the surface $\Si=x(\Sb)$. 
	(This operator $\grad ^N$ depends on $x$.)
	Then by the first variation
	formula of the Willmore energy (see e.g. \cite{MR2780248}*{Sec. 3}),
	this $\grad ^N \W(x)$ is given by the r.h.s. of \eqref{W}.
	This allows us to rewrite \eqref{W} as 
	$$\di_t x^N=\grad ^N \W(x)\quad (x\in X^k_c).$$

	Equivalently, \eqref{W} is the (negative) $L^2$-gradient flow of the Hawking mass,
	\begin{equation}
		\label{haw}
		m_\text{Haw}(\Si):=\frac{\abs{\Si}^{1/2}}{(16\pi)^{3/2}}\del{16\pi-\frac{1}{2}\int_\Si H^2\,d\mu_\Si^g},
	\end{equation}
	in the sense that a flow of surfaces 
	evolving according to \eqref{W} increases the mass $m_\text{Haw}$.
	For interests from physics related to  this problem, especially in general relativity,  
	see \cite{MR645761}.
	
	\subsection{Preliminary results}
	
	Let $R\gg1$ be given. Let  $K\subset M$ be a fixed compact set as in \secref{sec:1.1}.
	As explained in the last subsection, for asymptotically Schwarzschild manifold $M$,
	we can identify 
	the $M\setminus K$ with its coordinate space
	$\Rb^3\setminus\overline{B_R(0)}$.
	
	Let $\delta>0$ be given. 
	\begin{definition}[Admissible surfaces]
		\label{defn1.1}
		For a closed surface $\Si\subset M$, define the inner and outer radii
		$\rho(\Si),\,\lambda(\Si)$ as
		\begin{align}
			\label{1.4} 
			\rho(\Si)=&\min_{x\in \Si}\abs{x},\\
			\label{1.5}
			\lambda(\Si):=&\sqrt{\abs{\Si}/4\pi}.
		\end{align}
		We say $\Si$ is admissible if the interior of $\Si$ contains the fixed compact set
		$K$, 
		and 
		\begin{equation}\label{1.6}
			\rho(\Si)> R,\quad \abs{\frac{\rho(\Si)}{\l(\Si)}-1}+\int_\Si |\Acirc|^2<\delta.
		\end{equation}
		Here, recall, $\Acirc$ denotes the traceless part of the second fundamental form on $\Si$.
	\end{definition}
	\begin{remark}
		Geometrically, a surface $\Si$ is admissible if the origin lies sufficiently deep inside
		the interior of $\Si$ (this property is  called \textit{centering} in \cite{eichmair2021large}), 
		and at the same time the surface does not wiggle too much.
		It follows from the definition \eqref{1.5} that  $\lambda(\Si)\le  \max_{x\in \Si}\abs{x}$.
		Using the terminology in  \cite{eichmair2021large}, every admissible surface $\Si$
		satisfying \eqref{1.6} with  $R,\delta^{-1}\gg1$ is 
		\textit{on-center}.
	\end{remark}

	For the class of admissible surfaces, we have the following well-posedness result for
	\eqref{W}:
	\begin{theorem}[\cite{MR4236532}*{Thm. 5.3}]
		\label{thm1}
		Assume $M$ is $C^4$-close to Schwarzschild and satisfies 
		\eqref{1.2}-\eqref{1.3}. 
		Then for $R\gg1,\,\delta\ll1$ and every
		admissible surface $\Si_*$ satisfying \eqref{1.6}, there exists a  global solution to \eqref{W}
		with initial configuration $\Si|_{t=0}=\Si_*$.
	\end{theorem}

	Recall that Stationary solutions to \eqref{W} are called 
	surfaces of Willmore type.
	The existence and stability of such surfaces are 
	studied in \cites{eichmair2021large, MR4236532}. 
	\begin{theorem}[\cite{MR2785762}*{Thm. 1}, \cite{MR4236532}*{Thm. 5.3}]
		\label{thm2}
		Assume $M$ is $C^4$-close to Schwarzschild and satisfies 
		\eqref{1.2}-\eqref{1.3}.
		Then there exists a compact subset $K\subset M$
		such that $M\setminus K$ is foliated by surfaces of Willmore type.
		
		Moreover, for $R\gg1,\,\delta\ll1$ and every
		admissible surface $\Si_*$ satisfying \eqref{1.6},
		the flow generated by $\Si_*$ under \eqref{W} converges
		smoothly to one of the leaves of this foliation.
		
	\end{theorem}
	
	\subsection{Organization of the paper}

	We organize this paper as follows: In Section 2, we define the important map $\Phi$,
	which arises by reconceptualizing the Lyapunov-Schmidt reduction.
	This allows us to identify the adiabatic part of a flow evolving 
	according to \eqref{W}. This adiabatic part accounts for most of the (Willmore) energy change along the flow (modulo some uniformly
	small fluctuation),
	and is  finite-dimensional.
	In Section 3, we discuss the static property of the effective action $\W\circ \Psi$,
	and prove the first part of \thmref{thm3}. A similar but different 
	function defined on a domain in $\Rb^3$ 
	is used in \cite{eichmair2021large} (denoted by $G$ in that paper). 
	In Section 4, we prove the remaining part of \thmref{thm3} by deriving the effective
	dynamics \eqref{2} of \eqref{W}. Here we exploit the spectral property of certain
	linearized operators, in order to bound a Lyapunov-type functional that controls fluctuations.  
	
	\subsection*{Notation}
	
	Throughout the paper, the notation $A\ls B$ means that there is a constant $C>0$ depends
	only on $c,k$ in \eqref{1} and $R,\delta$ in \eqref{1.6}, such that $A\le CB$. 
	For two vectors $A,\,B$ in Banach spaces $X,\,Y$,
	the notation $A=O_{Y}(B)$ means $\norm{A}_{X}\ls \norm{B}_{Y}$. 
	For a vector $A$ in some Sobolev space $H^k$, the notation $A=O_{Y}(B)$ means 
	$\norm{A}_{H^k}=O_Y(B)$.

	\section{The Lyapunov-Schmidt map}
	\label{sec2}
	Let $k\ge 4,c\gg 1$. Let $K\subset M$, $R\gg0$ to be determined, and let
	$M':= \Rb_{>R}\times B_1(0)\subset \Rb\times \Rb^3$.
	In this section we construct the map $\Psi:M'\to X^k$ as in
	\thmref{thm3}.
	
	\subsection{Graphs over sphere}
	
	
	Denote $H^k=H^k(\Sb,\Rb)$. This space is equipped with the $L^2$-inner product 
	$\inn{u}{v}=\int_\Sb uv$.  Define the configuration space
	\begin{equation}
		\label{2.1}
		Y^k:=H^k\times M'.
	\end{equation}
	Define a map
	\begin{equation}
		\label{3}
		\fullfunction{\theta}{Y^k}{X^k}{(\phi,r,z)}{r(1+\phi(v))v+z}.
	\end{equation}
	Here $v\in \Sb\subset  \Rb^3$ is the spherical coordinate,
	and recall we identify the asymptotic part $(M\setminus K)\cong (\Rb^3\setminus B_R(0))$.
	Define 
	\begin{equation}
		\label{2.1'}
		Y^k_c:=\Set{(\phi,r,z)\in Y^k:\theta(\phi,r,z)=c}.
	\end{equation}
	This corresponds to the space of surfaces with fixed area, $X^k_c$, as in \eqref{1}. 
	
	For $\norm{\phi}_{H^k}\ll1$, the map $\theta(\phi,r,z)$ is a well-defined graph over 
	the coordinate sphere $\theta(0,r,z)(\Sb)=:S_{r,z}$. 
	Thus we can also identify $\theta(\phi,r,z)$ as
	a function from $S_{r,z}\subset M\to \Rb$. 
	Note also that for sufficiently large $c\gg1$ and every $z\in B_1(0)\subset \Rb^3$, there is a coordinate sphere with area $c$ around
	$z$. Thus the map $\theta$ is surjective onto $X^k_c$. 
	
	
	\begin{definition}[topology on graphs]
		\label{defn1}
		We say two graphs $\theta(\phi,r,z),\,\theta(\phi',r',z')$ are $H^k$-close
		if $\norm{\phi-\phi'}_{H^k}+\abs{r-r'}+\abs{z-z'}\ll1$.
	\end{definition}
	
	\subsection{Lyapunov-Schmidt reduction} \label{sec:2.2}

	\label{lem1}
	Denote $\b{W}(\phi,r,z),\,\Omega(\phi,r,z)$ the pullbacks of the r.h.s. of \eqref{W}
	and the Willmore energy \eqref{WE} to $Y^k$ through $\theta$, respectively.
	Explicitly, we have 
	\begin{align}\label{2.2}
		\b{W}(\phi,r,z):=&-W(\theta(\phi,r,z))-\lambda H(\theta(\phi,r,z)),\\
		\Omega(\phi,r,z):=&\W(\theta(\phi,r,z)). \label{2.2.1}
	\end{align}
	Since $\W$ is $C^2$ on $X^k$ with $k\ge 4$ and $\theta$ is smooth,
	the pullback energy $\Omega$ is $C^2$ on $Y^k,\,k\ge4$.
	Using Sobolev inequalities,
	one can check that the partial Fr\'echet derivative
	$\b{W}$ is $C^1$ in $\phi$ and smooth in $r,\,z$. 
	This $\b W$ is  the $L^2$-gradient of $\Omega(\cdot,r,z)$ up to scaling, and
	satisfies the mapping property $\b{W}:Y^k\to H^{k-4}$. 
	
	\begin{remark}
		Notice that \eqref{2.2}-\eqref{2.2.1} both implicitly depend  on the background metric $g$.
	\end{remark}
	
	\begin{lemma}
		The linearized operator $L_{r,z}$ of $\b W$ at $(0,r,z)$ with background metric $g$ is given
		by
		\begin{equation}
			\label{4}
			\begin{split}
				L_{r,z}^g:=&\di_\phi\b{W}(\phi,r,z)\vert_{\phi=0}\\
				=&(\Lap^2+2r^{-2}\Lap+O(r^{-4}))\di_\phi\theta(0,r,z):H^k\to 
				H^{k-4}.
			\end{split}
		\end{equation}
		Here $\Lap:X^k\to X^{k-2}$ denotes the Laplace-Beltrami operator  on the coordinate sphere $S_{r,z}\subset M\setminus K$, with center  $z$
		and radius $r$.
		The partial  Fr\'echet derivative $\di_\phi\theta(0,r,z): H^k\to X^k$ is given by 
		$\xi(v)\mapsto \xi(v)rv$. 
		
		Moreover, the operator $L_{r,z}$ is self-adjoint on $H^k$.  The spectrum of $L_{r,z}$
		is purely discrete. The operator $\di_\phi\theta(0,r,z)$ is invertible and satisfies 
		\begin{equation}
			\label{2.6}
			\norm{\di_\phi\theta(0,r,z)}_{H^k\to X^k}= \norm{\di_\phi\theta(0,r,z)^{-1}}^{-1}_{X^k\to H^k}=r.
		\end{equation}
	\end{lemma}
	\begin{proof}
		The  operator $L_{r,z}^g$ is 
		explicitly calculated in \cite{MR2785762}*{Sec. 3}.
		The spectral properties  of $L_{r,z}$ are studied in \cite{MR2785762}*{Sec. 7}.
		The mapping properties of $\di_\phi\theta$ is obvious.
	\end{proof}
	\begin{remark}
		\label{rmk1}
		The linearized operator \eqref{4} depends on (the curvature of ) the background metric $g$
		on $M$. In the special case when the ambient manifold $M$ is flat, i.e. $g=\delta_{ij}$, the
		linearized operator 	$L_{r,z}^0$ has eigenvalue $0$, and 
		$\ker L_{r,z}^0$ is spanned by the constant function $y^0\equiv 1$,
		together with the spherical harmonics $y^1,y^2,y^3$. Thus, so long as  $(M,g)$ is asymptotically flat
		and $r\gg1$ in \eqref{4} (such as in our setting),
		one can view $L_{r,z}^g$ as a perturbation of $L_{r,z}^0$. 
		This motivates the following definition.
	\end{remark}

	\begin{definition}
		\label{defn2}
		Define $P:H^k\to H^k$ to be the $L^2$-orthogonal projection onto  $\spn\Set{y^0,\ldots,y^4}=\ker{L_{r,z}^0}$. 
		Define $\b{P}:=1-P:H^k\to H^k$ be the complement of $P$. 
		
		Let $\mathcal{S}$ be the set of all smooth symmetric two tensors on $M$.
		Define a map
		\begin{equation}
			\label{5}
			\fullfunction{F}{Y^k\times \mathcal{S}}{H^{k-4}}{(\phi,r,z,h)}{\b{P}\b{W}(\phi,r,z)},
		\end{equation}
		where $\b W$ is computed with  background 
		metric $g=g_S+h$ (see \secref{sec:1.1}).
	\end{definition}

	\begin{proposition}
		\label{prop1}
		Assume the ambient manifold $(M,g)$ is $C^k$-closed to Schwarzschild.
		
		\begin{enumerate}
			\item 	For every $z\in \Rb^3$ with $\abs{z}<1$ and
			sufficiently large $r\ge R\gg1 $, 
			there is a unique solution $\phi=\phi_{r,z}\in \b{P}H^k$
			to the equation
			\begin{equation}
				\label{6}
				F(\phi,r,z,h)=0,
			\end{equation}
			where $F$ is defined in \eqref{5}, and $g=g_S+h$. 
			
			\item Moreover, the map $(r,z)\mapsto \phi_{r,z}$ is $C^2$, and satisfies the estimate
			\begin{equation}
				\label{7}
				\norm{\di_r^m\di^\al_z\phi_{r,z}}_{H^k}\ls r^{-(2+2m)}
			\end{equation}
			for every $m+\abs{\al}\le 2$. 
			
			
			\item Moreover, the surface $\theta(\phi_{r,z},r,z)$ lies in the class of admissible surfaces
			in \defnref{defn1.1}
		\end{enumerate}
	\end{proposition}
	\begin{proof}
		1. By the Implicit Function Theorem, it suffices to check that
		the map $F$ defined in \eqref{6}
		satisfies the following properties:
		\begin{enumerate}
			\item $F$ is $C^1$ in $\phi$.
			\item $F(0,r,z,0)=0$ for every $r,z$.
			\item $\di_\phi F(0,r,z,0)=L_{r,z}^0$ is invertible on $\b{P}H^k$. 
		\end{enumerate}
		The first claim follows from the regularity of $\b{W}$ on $Y^k$
		and its smooth dependence on the background metric.
		
		If the background metric is Schwarzschild, i.e. $h=0$, then
		it is well-known that by conformal invariance the coordinate sphere $\theta(0,r,z)$
		is the global minimizer of 
		the Willmore energy $\W$. 
		Since $\bar{W}=\di_\phi\Omega$ (see \eqref{2.2.1}), the second claim follows. 
		
		The spectrum of $L_{r,z}^0$ can be calculated explicitly. See for instance
		\cite{eichmair2021large}*{Cor. 33}.  	
		In particular, $0$ is an isolated eigenvalue with
		finite multiplicity. By elementary spectral theory, this implies 
		the restriction $\b L_{r,z}^0:= L_{r,z}^0\vert_{\b P}$ is invertible as a map from $\b PH^k\to \b PH^k$.
		Thus the third claim follows.
		
		2. For the estimate \eqref{7}, we expand \begin{equation}\label{V}
			L_{r,z}^g=L_{r,z}^0+V_{r,z},
		\end{equation}
		where $V_{r,z}$ 
		is defined by this expression.
		As we discuss in \remref{rmk1}, this $V_{r,z}$ is bounded from $H^k\to H^{k-4}$, 
		and satisfies $\norm{V_{r,z}}_{H^k\to H^{k-4}}=O(r^{-4})$. 
		The restriction $\b L_{r,z}^0$ can be bounded from below by $Cr^{-2}$
		for some $C>0$ only depending on $k$.
		It follows  that
		$$\norm{(\b L_{r,z}^0)^{-1}V_{r,z}}_{H^{k-4}\to H^k}=O(r^{-2}).$$
		For sufficiently large $r$, this together with the expansion \eqref{V}
		implies that the restriction
		$\b L_{r,z}^g:\b PH^k\to\b P H^{k-4}$ is invertible, 
		given explicitly as the Neumann series
		$$(\b L_{r,z}^g)^{-1}= \sum_{n=0}^\infty (\b L_{r,z}^0)^{-1} (-V_{r,z}(\b L_{r,z}^0)^{-1})^n.$$
		From here one can also read off the estimate
		\begin{equation}
			\label{2.7}
			\norm{(\b L_{r,z}^g)^{-1}}_{\b PH^{k-4}\to\b P H^k}=O(r^2).
		\end{equation}
		
		Expand $F(\phi,r,z,h)=F(0,r,z,h)+\b L_{r,z}^g\phi+N_{r,z}(\phi)$,
		where the nonlinearity  $N_{r,z}$ is defined by this expression.
		This $N_{r,z}$ is calculated explicitly in \eqref{B1.7}.
		For every $\phi$ satisfying  \eqref{6}, we can rearrange to get 
		\begin{equation}
			\label{2.8}
			\phi=-(\b L_{r,z}^g)^{-1}(F(0,r,z,h)+ N_{r,z}(\phi)).
		\end{equation}
		In the r.h.s. we have $F(0,r,z,h)=O(r^{-4})$ by \cite{eichmair2021large}*{Cor. 45}.
		Thus, for sufficiently small $\phi$,  we have by \eqref{2.7}-\eqref{2.8} that 
		$\norm{\phi}_{H^k}=O(r^{-2})$. 
		
		We now claim for $\phi\in H^k$ and $m+\abs{\al}\le2$, there hold
		\begin{align}
			\label{2.9}
			\norm{\di_r^m\di^\al_z(\b L_{r,z}^g)^{-1}\phi}_{H^k}\ls& \norm{\phi}_{H^{k-4}},\\
			\label{2.10}
			\norm{\di_r^m\di^\al_zF(0,r,z,h)}_{H^{k-4}}\ls& r^{-(4+m)},\\
			\label{2.11}
			\norm{\di_r^m\di^\al_zN_{r,z}(\phi)}_{H^{k-4}}\ls& \norm{\phi}_{H^k}^2.
		\end{align}
		For \eqref{2.9}, one uses the identity $\di^\beta (\b L_{r,z}^g)^{-1}=-(\b L_{r,z}^g)^{-1}\b \di^\beta  L_{r,z}^g (\b L_{r,z}^g)^{-1}$, where $\abs{\beta}\le2$ is a multi-index
		in both $r$ and $z$. This, together with the fact that $\di^\beta  L_{r,z}$ is uniformly bounded
		(see \eqref{B1.5}),
		implies \eqref{2.9}. 
		The rest follows from the expansion in \propref{propB2}.
		Using \eqref{2.9}-\eqref{2.11} and differentiating both sides of \eqref{2.8},
		we conclude the estimates \eqref{7}.
		
		3. For sufficiently large $R$ and every $r\ge R$,
		we find using \eqref{7} with $m=0,\al=0$ that the surface $\theta(\phi_{r,z},r,z)$ 
		is $H^k$-close to the coordinate sphere $S_{r,z}$. 
		This implies $\theta(\phi_{r,z},r,z)$ is an admissible surface.
		
	\end{proof}

	From now on we write $\zeta=\z^\al,\,\al=0,\ldots,4$, for a point in $(r,z)\in  M'$.
	Thus, $\z^0=r$ and $\z^j=z^j$ for $j=1,2,3$.

	\begin{definition}[The Lyapunov-Schmidt map $\Phi$]
		\label{LSM}
		Let $K\subset M$ be the compact set as in \thmref{thm2}.
		Let $R\gg1,\,\delta\ll1$ be given as in \thmref{thm1}.
		Let $M':= \Rb_{>R}\times B_1(0)\subset \Rb\times \Rb^3$.
		
		Define \textit{the Lyapunov-Schmidt map} $\Phi:M'\to H^k$ by $\z\mapsto\phi_\z$, where 
		$\phi_\z$ 
		is the solution to \eqref{6} given in \propref{prop1}.
	\end{definition}
	\begin{remark}
		This $\Phi$ is equivalent to the map $\Psi$ in \thmref{thm3}, through
		the diffeomorphism $\Phi(\z)\mapsto \theta(\Phi(\z),\z)$.
	\end{remark}
	
	In the next proposition, we describe the geometric structure 
	induced by the map $\Phi$. 
	\begin{proposition}
		\label{prop2}
		The set  $$E:=\Set{\theta(\phi,\z):\phi=\Phi(\z),\z\in M'}$$ forms an immersed  $C^1$ submanifold in $X^k$.
		The tangent space $T_{\theta(\Phi(\z), \z)}E$ consists of vector fields over the surface
		$\theta(\Phi(\z), \z)(\Sb)$. A basis of $T_{\theta(\Phi(\z), \z)}E$ is given by
		$\di_{\z^\al}\theta(\Phi(\z), \z)$.
	\end{proposition}
	\begin{remark}
		Using the projection constructed in \lemref{lem2.1}, one can view this manifold $E$ as
		consisting of the adiabatic parts of low (Willmore) energy surfaces 
		in $X^k$. 
	\end{remark}
	
	\begin{proof}
		The manifold structure of $E$ follows from \defnref{LSM}, where $\Phi:M'\to E$
		is a $C^1$ parametrization. We check the tangent space 
		is non-degenerate. 
		Compute 
		\begin{align}
			\label{2.15}\di_{\z^0}\theta(\Phi(\z), \z)(v)=&(1+\Phi(\z)+\z^0\di_{\z^0}\Phi(\z))v,\\
			\label{2.16}
			\di_{\z^j}\theta(\Phi(\z), \z)(v)=&\z^0\di_{\z^j}\Phi(\z)v+e^j,
		\end{align} 
		where $e^j$ is the $j$-th unit vector in $\Rb^3$.
		By the estimate \eqref{7},
		we find $$\inn{\di_{\z^\al}\theta(\Phi(\z), \z)}{\di_{\z^\beta}\theta(\Phi(\z), \z)}=4\pi\delta_{\al\beta}+O(R^{-2}).$$
		This implies the claim if $R$ is sufficiently large.
	\end{proof}
	
	In Appendix, we introduce the general concepts of the Lyapunov-Schmidt map, and relate it to
	our setting above.

	\subsection{Barycenter}

	In this subsection, we develop a new concept of barycenter for a certain class
	of closed surfaces in $X^k$. 
	
	\begin{definition}[Barycenter]
		\label{bar}
		Let $x_*$ be an embedding of sphere that is $H^k$-close to the manifold 
		$E\subset X^k$ constructed in \defnref{LSM}, w.r.t. the topology on graphs introduced in \defnref{defn1}.
		Then we can write $x_*=\theta(\Phi(\z)+\xi,\z)$
		for some $\z\in M', \norm{\xi}_{H^k}\ll1$. (There can in general be many 
		such choice of $\z$ and $\xi$.)
		Expand $x_*$ in $\xi$ around $\theta(\Phi(\z),\z)$ as 
		\begin{equation}
			\label{4.1}
			x_*=\theta(\Phi(\z),\z)+\di_\phi\theta (\Phi(\z),\z)\xi+O(\norm{\xi}_{H^k}^2).
		\end{equation}

		Define $f_\al\in H^k$ as
		\begin{equation}\label{4.2}
			\begin{split}
				f_\al(\z)(v)=&\di_{\z^\al}\theta(\Phi(\z),\z)(v)^N\\=&g(\di_{\z^\al}\theta(\Phi(\z),\z)),\nu(\theta(\Phi(\z),\z))\quad (\al=0,\ldots,3).
			\end{split}
		\end{equation}

		We say a point $\z_*\in M'$ is the barycenter of $x_*$ if
		$\z_*$ solves the following algebraic system:
		\begin{equation}
			\inn{\xi}{f_\al}_{L^2}=0\quad(\al=0,\ldots,3), \label{4.4}
		\end{equation}
		where $\xi$ is defined by the relation $x_*=\theta(\Phi(\z_*)+\xi,\z_*)$. 
	\end{definition}
	\begin{remark}
		\label{rmk2}
		The four vectors $f_\al$ span the 
		tangent space at
		$\theta(\Phi(\z_t),\z_t)^N$ to $E^N\subset H^k$, where
		$E^N$ consists of the normal components of the elements in
		the manifold $E$ defined in  \defnref{LSM}. 
		Geometrically, the defining condition \eqref{4.4} for barycenter
		means  that the G\^ateaux derivative
		of the map $\theta(\cdot,\z_*)^N$ at $\Phi(\z_*)$ 
		along $\xi$-direction is perpendicular to the 
		tangent space $T_{\theta(\Phi(\z_*),z_*)^N}E^N$.
		In terms of the expansion \eqref{4.1}, 
		this means the second term in the r.h.s. is $L^2$-orthogonal 
		to the tangent space at the first term to $E$. 
		In this sense, the choice of barycenter is optimal.
	\end{remark}
	\begin{remark}
		Our definition of barycenter differs from the classical one, given by
		averaging over $\Si$ w.r.t. Euclidean background metric, namely $\abs{\Si}_g^{-1}\int_\Si xd\mu_\Si^{\delta_{ij}}$.
		See \cite{MR4236532} and the references therein. 
		Our version of barycenter retains the key decay property as 
		\cite{MR4236532}*{Sec. 5}. Namely, the motion of barycenter
		is controlled by a differential inequality using a Lyapunov
		functional, defined in \secref{sec:4}.
		
		Moreover, our definition allows us to retain explicit and  uniform control of 
		a flow evolving according to \eqref{W}, as we show in Sec. \ref{sec:4}.
	\end{remark}
	
	In the next lemma, we define a nonlinear projection (or coordinate map)  that associates 
	barycenters to low energy configurations in $X^k$. 
	\begin{lemma}[nonlinear projection]
		\label{lem2.1}
		There exists $\delta>0$ such that
		on the space \begin{equation}
			\label{2.3}
			U_\delta:=\Set{x=\theta(\Phi(\z)+\xi,\z):\z\in M', \,\norm{\xi}_{H^k}<\delta},
		\end{equation}
		there exists a $C^1$ map $S:U_\delta\to M'$ such that $S(x)$ is the barycenter of $x$
		as in \defnref{bar}.
		
		Moreover, we have uniform estimate on $S$ and its derivative. 
	\end{lemma}
	\begin{remark}
		Essentially, the existence of such projection depends on the non-degeneracy shown in \propref{prop2}. 
		Later, 
		we see that the barycenter $\z=S(x)$ determines the adiabatic (or slowly-varying) part of $x$.
		
		We call the remainder $\xi$ that satisfies $x=\theta(\Phi(S(x))+\xi,S(x))$ the \textit{fluctuation}
		of $x$. 
	\end{remark}
	\begin{proof}
		Define a map 
		\begin{equation}
			\label{2.4}
			\fullfunction{\Ga}{M'\times U_\delta\subset \Rb^4\times H^k}{\Rb^4}{(\z,x)}{\inn{\xi}{f_\al}_{L^2}},
		\end{equation}
		where $\xi$ is defined by the relation $x=\theta(\Phi(\z)+\xi,\z)$.
		It suffices to find a map $S$ such that $\z=S(x)$ solves 
		\begin{equation}
			\label{2.5}
			\Ga(\z,x)=0.
		\end{equation}
		
		By the Implicit Function Theorem, it suffices to check that
		the map $\Ga$ 
		satisfies the following properties:
		\begin{enumerate}
			\item $\Ga$ is $C^1$ in $\z$.
			\item $\Ga(\z,\theta(\Phi(\z),\z))=0$.
			\item The matrix $\grad_\z \Ga(\z,\theta(\Phi(\z),\z)):\Rb^4\to \Rb^4$ is invertible. 
		\end{enumerate}
		The first claim follows from the regularity of $\Phi$ as in \propref{prop1}.
		The second claim is trivial because in this case $\xi=0$.
		
		We now claim the rescaled matrix $$ A_{\al\beta}:=\z^0\di_{\z^\al} 
		\Ga_\beta\vert_{(\theta(\Phi(\z),\z))}$$ is invertible, which implies the third claim above.
		
		Using \eqref{2.15}-\eqref{2.16},
		the definition \eqref{4.2}, and the assumption on the background metric $g_{ij}=\delta_{ij}+O((\z^0)^{-1})$, we compute
		\begin{align}
			\label{2.17}
			f_0=&(1+O((\z^0)^{-1}))y^0+\Phi(\z)+\z^0\di_{\z^0}\Phi(\z)+\Phi(\z)O((\z^0)^{-1})+O((\z^0)^{-4}),\\
			\label{2.18}
			f_j=&(1+O((\z^0)^{-1}))y^i+\z^0\di_{\z^j}\Phi(\z)+O((\z^0)^{-4}),
		\end{align}
		where the vectors $y^\al$ span $\ran P$ as in \defnref{defn2}.
		
		For $\xi=\xi(\z,x)$, we find 
		\begin{align}
			\label{2.12}
			\xi(v)=&\inn{\frac{x(v)-\z}{\z^0}}{v}-1-\Phi(\z)(v),\\
			\label{2.13}
			\di_{\z^0}\xi(v)=&-\inn{\frac{x(v)-\z}{(\z^0)^2}}{v}-\di_{\z^0}\Phi(\z),\\
			\label{2.14}
			\di_{\z^j}\xi(v)=&-\frac{e^j}{\z^0}-\di_{\z^j}\Phi(\z)(v).
		\end{align}
		Now, since $x\in U_\delta$, 
		we have $A_{\al\beta}=\z^0\inn{\di_{\z^\al}\xi}{f_\beta}+\z^0\inn{\xi}{\di_{\z^\al} f_\beta}=\inn{\di_{\z^\al}\xi}{f_\beta}+O(\z^0\delta)$. By this, and the formula
		\eqref{2.17}-\eqref{2.14} above, we find that $A_{\al\beta}=O(1)\delta_{\al\beta}+O((\z^0)^{-2})+O(\z^0\delta)$.
		For sufficiently large $R\gg1$, $\delta=o(R^{-1})$, and every $\z^0\ge R$, we can conclude from here
		that $A_{\al\beta}$ is invertible. This proves the existence of the $C^1$ map $S$.
		The uniform estimates for $S$ and its Fr\'echet derivative are implicit in the arguments above.

	\end{proof}
	
	
	\section{Effective action}\label{sec:3}
	In this section we prove Part 1 of \thmref{thm3}.
	We formulate this as follows:
	\begin{theorem}
		\label{thm4}
		The embedding $\theta(\Phi(\z),\z)$	 parametrizes a surface of Willmore type (i.e.  static solution  to \eqref{W}) if and only if $\z$
		is a critical point of the function $G:=\Omega(\Phi(\cdot ),\cdot):\Rb^4\to \Rb$, where $\Omega$ is defined in \eqref{2.2.1}.
	\end{theorem}
	\begin{remark}
		Similar results are obtained in \cite{eichmair2021large}*{Thms. 5,8}. Using some 
		expansion obtained in that paper, we can calculate $G$ explicitly as in \eqref{B1.2}.
	\end{remark}
	
	\begin{proof}
		For the forward direction, we use  the chain rule to get 
		\begin{equation}
			\label{3.1}\di_{\z^\al} G(\z) =\inn{\grad ^N\W(\theta(\Phi(\z),\z))}{f_\al}.
		\end{equation} If $\theta(\Phi(\z),\z)$ is a 
		critical point of $\W$, i.e. the Fr\'echet derivative
		$d\W(\theta(\Phi(\z),\z))=0$, then the first factor of 
		r.h.s. of \eqref{3.1} vanishes, and therefore
		$\z$ is a critical point of $G$.
		
		Now, suppose $\di_{\z^\al} G(\z) =0$. For the backward direction,
		it suffices to show the pullback
		$\b W (\Phi(\z),\z)=0$.
		
		Let $Q_{\z}$ be the projection onto the tangent space at
		$T_{\theta(\Phi(\z_t),\z_t)^N}E^N$, as in \remref{rmk2}.  
		Explicitly, this map $Q_\z$ is given by
		\begin{equation}
			\label{3.2}
			Q_\z \phi=V^{\al\beta}\inn{f_\al}{\phi}_{L^2}f_\beta\quad (\phi\in H^k),
		\end{equation}
		where the matrix $V_{\al\beta}:= \inn{f_\al}{f_\beta}$, the matrix
		$V^{\al\beta}$ is its inverse, and the $f_\al$'s are given in \eqref{2.17}-\eqref{2.18}.
		Notice that this matrix $V_{\al\beta}$ is indeed invertible because 
		the tangent space $T_{\theta(\Phi(\z_t),\z_t)^N}E^N$ is non-degenerate, c.f.
		\propref{prop2}.
		From the definition \eqref{3.2} and the formula \eqref{2.17}-\eqref{2.18},
		we also get an uniform estimate $\norm{Q_{\z}}\ls 1$.
		
		The claim now is that $Q_\z:H^k\to H^k$ is uniformly close in operator norm 
		to the projection $P$ defined in \defnref{defn2}.
		Geometrically, this is because the manifold $E^N$ is a 
		uniformly small normal perturbation of $\ran P$. Analytically,
		this claim follows by comparing \eqref{2.17}-\eqref{2.18} to 
		the basis $y^\al$ of $\ran P$, and using the estimate \eqref{7}
		for the derivatives $\di_{\z^\al}\Phi$.
		As a result, we find
		\begin{equation}
			\label{3.2.1}
			\norm{Q_\z-P}_{H^k\to  H^k}=O(R^{-2}).
		\end{equation}
		
		Now, by the construction in \defnref{LSM}, we know $P\b W(\Phi(\z),\z)=\b W(\Phi(\z),\z)$.
		Thus we can write 
		\begin{equation}
			\label{3.3}
			\b  W(\Phi(\z),\z)=P \b W(\Phi(\z),\z) = (P-Q_\z)\b W(\Phi(\z),\z)+ Q_\z \b W(\Phi(\z),\z).
		\end{equation}
		From \eqref{3.1}-\eqref{3.2}, one can see that  $\di_{\z^\al} G=0$  implies $Q_\z\b W=0$.
		Thus by assumption, the last term in \eqref{3.3} vanishes. Using this fact and the estimate \eqref{3.2.1}, we find
		$$\begin{aligned}
			\norm{\b W(\Phi(\z),\z)}_{H^k} &\ls \norm{(P-Q_\z)}_{H^k\to H^k}\norm{\b W(\Phi(\z),\z)}_{H^k}\\ &\ls R^{-1}\norm{\b W(\Phi(\z),\z)}_{H^k}.
		\end{aligned}$$
		For sufficiently large $R$, this is impossible unless $\b W(\Phi(\z),\z)=0$.

	\end{proof}

	\section{Effective dynamics}
	\label{sec:4}
	In this section we prove the rest of \thmref{thm3}.
	We first derive the effective dynamics \eqref{2.0}-\eqref{2.0.1}, 
	and then use this to 
	derive Parts 2 and 4 of \thmref{thm3}.
	
	\begin{theorem}
		\label{thm5}
		Let $\Phi:M'\to H^k$ be the map defined in \defnref{LSM}.
		Let $\Si_*=x_*(\Sb)$ be an admissible surface. Let $\Si_t=x_t(\Sb)$
		be the global solution to \eqref{W}
		with initial configuration $\Si|_{t=0}=\Si_*$ as in \thmref{thm1}.
		
		Then there exist $\al>0$,  $T=O(R^{-\al})$, and
		a path $\z_t\in M'$, such that
		\begin{equation}
			\label{4A}
			\norm{\Phi(\z_t)-x_t^N}_{H^k}=O(R^{-3})\quad (t\ge T).
		\end{equation}
		
		Moreover, the path $\z_t$ evolves according to 
		\begin{equation}
			\label{4B}
			\dot \z=\frac{1}{4\pi}\grad G(\z)+O(R^{-3}),
		\end{equation}
		where the leading term in the r.h.s.  is of the order $O(R^{-2})$.
	\end{theorem}
	\begin{remark}
		The function $G$ is defined in \thmref{thm4} and calculated in Appendix \ref{sec:A2}.
	\end{remark}
	
	\begin{proof}

		Step 1. 
		Take some $0<\delta\ll1$ to be determined.
		To begin with, denote $T>0$ the first time 
		$u_t$ enters the manifold $U_\delta$ defined in \eqref{2.3}.
		By the stability result \thmref{thm2}, this $T$ is finite,
		and there exists
		$\al>0$ such that  $T=O(\delta^{-\al})$.

		Now we prove  the claims \eqref{4A}-\eqref{4B},
		assuming the a priori estimate 
		\begin{equation}
			\label{A}
			x_t\in U_\delta \text{ for all $t>T$}.
		\end{equation}
		By \lemref{lem2.1}, so long as 
		\eqref{A} is satisfied,
		we can associate a path of barycenters $\z_t\in M'$ to the full flow $x_t$. 
		Later on, we show this Ansatz is satisfied with $\delta=O(R^{-3})$ by
		proving an a priori estimate for the fluctuation of $x_t$ 
		around its adiabatic part.

		Write $x=\theta(\Phi(\z)+\xi,\z)$, where
		$\z=S(x)$ is the barycenter as in \lemref{lem2.1}, and $\xi$ is the fluctuation.
		Then we can rewrite 
		the l.h.s. of \eqref{W} as
		\begin{equation}
			\label{4.5}
			\begin{aligned}
				\di_tx^N&=\di_\xi\theta(\Phi+\xi,\z)^N\di_t\xi+ \di_{\z^\al}\theta(\Phi+\xi,\z)^N\dot{\z^\al}\\
				&=A(\xi)\di_\phi\theta(\Phi,\z)^N\di_t\xi+B(\xi) \di_{\z^\al}\theta(\Phi,\z)^N\dot{\z^\al}\\
				&=A(\xi)(\z^0+\OO{1})\di_t\xi+B(\xi)f_\al\dot\z^\al,
			\end{aligned}
		\end{equation}
		where we define
		$$A(\xi):=\frac{\di_\xi\theta(\Phi+\xi,\z)^N}{\di_\phi\theta(\Phi,\z)^N},\quad
		B(\xi):=\frac{\di_{\z^\al}\theta(\Phi+\xi,\z)^N}{\di_{\z^\al}\theta(\Phi,\z)^N}.$$
		For simplicity, here and below we omit the dependence of $f_\al$ and $\Phi$ on $\z$.
		
		The expansion  \eqref{4.5} follows from the chain rule, \eqref{4.2}, and the assumption on the background metric,
		$g_{ij}=\delta_{ij}+O((\z^0)^{-1})$.
		From \eqref{3}, one can see that the prefactors $A(\xi),B(\xi)$ satisfy
		\begin{equation}
			\label{4.5.1}
			\norm{A(\xi)-1}_{H^k}+\norm{B(\xi)-1}_{H^k}\ls \norm{\xi}_{H^k}.
		\end{equation}


		Let $Q=Q_{\z(t)}$  be the projection onto the tangent space
		$T_{\theta(\Phi(\z_t),\z_t)^N}E^N$, as given explicitly in \eqref{3.2}.
		The space $\ran Q$ is spanned by the four vectors $f_\al$,
		which we computed in \eqref{2.17}-\eqref{2.18}.
		Applying $Q$ to \eqref{4.5}, we find
		\begin{equation}
			\label{4.7}
			Q\di_tx^N=(\z^0+O_{H^k}(1))Q\di_t\xi +f_\al\dot \z^\al+\OO{\xi(\lvert\dot\z\rvert+\di_t\xi)}. 
		\end{equation}
		
		By the definition of barycenter, \eqref{4.4}, we know $Q\xi=0$. 
		Differentiating this,  we find 
		\begin{equation}
			\label{4.8}
			Q(\di_t\xi)=-(\di_tQ)\xi=-(\di_{\z^\al} Q\dot \z^\al)\xi.                                                       	
		\end{equation}         	
		Here, the partial Fr\'echet derivative  $\di_{\z^\al} Q$ maps a real number to a linear operator from $H^k\to \ran Q\subset H^k$. 
		From the explicit formula \eqref{3.2} for $Q$ and the estimate \eqref{7},
		we get the uniform bound $\norm{\di_{\z^\al} Q}_{\Rb\to L(H^k,H^k)}\ls (\z^0)^{-2}$. 
		Thus
		plugging \eqref{4.8} to \eqref{4.7} gives
		\begin{equation}
			\label{4.9}
			Q\di_tx^N=f_\al\dot\z^\al+\OO{\xi(\lvert\dot\z\rvert+\di_t\xi)}+R^{-1}\OO{\xi\lvert\dot\z\rvert}
		\end{equation}

		Step 2.
		Next, expanding the r.h.s. of \eqref{W} at $\theta(\Phi(\z),\z)$, we find
		\begin{equation}
			\label{4.10}
			\begin{split}
				W(x)+\l H(x)&=\b{W}(\Phi(\z)+\xi,\z)\\
				&=\b{W}(\Phi(\z),\z))+L_\z\xi +N_\z(\xi).
			\end{split}
		\end{equation}
		Here, as in Section \ref{sec2},  $\b{W}:Y^k\to H^{k-4}$ is the pullback of $W+\lambda H$, and $L_\z$
		is the partial Fr\'echet derivative $\di_\phi\b{W}$ evaluated at $\theta(\Phi(\z),\z)$.
		The nonlinear term $N_\z(\xi)$ is defined by collecting the rest terms in this expansion.
		
		Applying $Q$ to \eqref{4.10}, we find 
		\begin{equation}
			\label{4.11}
			Q\b{W}(\Phi(\z)+\xi,\z)=Q\b{W}(\Phi(\z),\z)+QL_\z\xi +QN_\z(\xi).
		\end{equation}
		By the chain rule and the definitions \eqref{2.2} \eqref{3.2}, 
		the first term in \eqref{4.11} can be written as
		\begin{equation}
			\label{4.12}
			Q\b{W}(\Phi(\z),\z)=V^{\al\beta}\inn{\grad ^N \W(\theta(\Phi(\z),\z))}{f_\al}f_\beta.
		\end{equation}
		By \eqref{3.1}, this expression equals to $V^{\al\beta}\di_{\z^\al} G(\z)f_\beta$.
		
		By the uniform estimates $\norm{Q}_{H^k\to H^k}\ls 1$ and \eqref{B1.6'}, we can
		bound the nonlinearity as 
		\begin{equation}
			\label{4.13}
			\norm{QN_\z(\xi)}_{H^{k-4}}=O(\norm{\xi}^2_{H^k}).
		\end{equation}
		
		Lastly, using the approximate zero modes property of the elements in $\ran Q$,
		which we show in \lemref{lemA1}, one sees   that the restriction of
		$L_\z$ is small on $\ran Q$. In particular, this  implies for
		$\xi,\xi'\in H^k$,
		$$
		\inn{QL_\z\xi}{\xi'}=\inn{\xi}{L_\z Q\xi'}= O(R^{-2})\norm{\xi}_{H^k}\norm{\xi'}_{H^{k-4}}.
		$$
		Plugging $\xi'=QL_\z\xi$ into this expression ,  we find 
		\begin{equation}
			\label{4.14}
			\norm{QL_\z\xi}_{H^{k-4}}= O(R^{-2})\norm{\xi}_{H^k}.
		\end{equation}
		
		
		Recall the Ansatz \eqref{A}, which implies $\norm{\xi}_{H^k}<\delta\ll1$.
		Collecting \eqref{4.9}, \eqref{4.11}-\eqref{4.14}, plugging these back to \eqref{4.5},\eqref{4.10}, and rearranging, we find
		\begin{equation}
			\label{4.15}
			\norm{f_\al\dot{\z^\al}-V^{\al\beta}\di_{\z^\al}G(\z)f_\beta}_{H^{k-4}}\ls\frac{R^{-2}O(\delta)+O\del{\delta\norm{\di_t\xi}_{H^k})}}{1-O(\delta)-R^{-1}O(\delta)}.
		\end{equation}
		
		
		Consider the map $$\fullfunction{\psi}{\Rb^4}{\ran Q\subset
			H^{k-4}}{\eta^\al}{V^{\al\beta}\eta^\al f_\beta}.$$
		This map is linear. Moreover, using the fact that the matrix $V^{\al\beta}=O(1)\delta_{\al\beta}$ is invertible, one can show  $\psi$ is invertible on $\ran Q$,  and
		the operator norm  of its inverse
		$\psi^{-1}:\ran Q\subset H^{k-4}\to \Rb^4$ is of the order $O(1)$.
		
		Now we rewrite 
		$f_\al \dot \z \al = \psi (\eta^\al)$ with 
		\begin{equation}
			\label{eta}
			\eta^\al = V_{\al\beta}\dot \z^\beta.
		\end{equation}
		Note that this choice of $\eta^\al$ is unique.
		Then we can conclude from \eqref{4.15} and the discussion above  that 
		\begin{equation}\label{4.15'}
			\abs{\eta^\al-\di_{\z^\al}G(\z)}\ls\frac{R^{-2}O(\delta)+O\del{\delta\norm{\di_t\xi}_{H^k})}}{1-O(\delta)-R^{-1}O(\delta)}.
		\end{equation}

		Step 3. From \eqref{4.15'}, the relation \eqref{eta}, and the estimate $V_{\al\beta}=4\pi\delta_{\al\beta}+O(R^{-2})$ which follows from \eqref{2.17}-\eqref{2.18}(see
		also \propref{prop2}),  
		we see that \eqref{4B} follows once we can
		prove an a priori estimate of the form 
		\begin{equation}
			\label{claim}
			\norm{\xi}_{H^k}+\norm{\di_t\xi}_{H^k}\ls\norm{\xi_T}_{H^k},
		\end{equation} where $\xi_T$ is the 
		fluctuation when the flow $x_t$ first enters the manifold $U_\delta$, as in 
		Ansatz \eqref{A}.

		To this end, define 
		\begin{equation}
			\label{4.21}
			\Lambda(t):=\frac{1}{2}\inn{\xi_t}{L_{\z_t}\xi_t}.
		\end{equation}
		We show this is a Lyapunov-type functional along the flow of $\xi$.
		
		Recall $\b Q=1-Q$ is the complement of the (not necessarily orthogonal) projection $Q$.

		Indeed, differentiating $\Lambda$, we find
		\begin{equation}
			\label{4.22}
			\dot \Lambda(t)=\inn{\di_t\xi}{L_{\z_t}\xi}+\frac{1}{2}\inn{(\di_tL_\z)\xi}{\xi}.
		\end{equation}
		The second term  is bounded as  
		\begin{equation}\label{4.16}
			\abs{\inn{(\di_tL_\z)\xi}{\xi}}\le \norm{(\di_\z L_\z\dot\z)\xi}_{H^{k-4}}\norm{\xi}_{L^2}=O(\lvert\dot\z\rvert\norm{\xi}_{H^k}^2),
		\end{equation} where
		we used the estimate \eqref{B1.5'}.
		
		To bound the first term in the r.h.s. of \eqref{4.22}, 
		we isolate the dynamics of $\xi$ after equating \eqref{4.5} to \eqref{4.10}. We find
		\begin{equation}
			\label{4.17}
			\di_t\xi = \frac{1}{\di_\phi\theta(\Phi,\z)^NA(\xi)}(-\b W(\Phi,\z)-L_\z\xi -N_\z(\xi)-B(\xi)f_\al\dot\z^\al).
		\end{equation}
		Consider the r.h.s. of \eqref{4.16}-\eqref{4.17}. Using \eqref{B1.4'}-\eqref{B1.6'} and the governing equation for barycenter, \eqref{4.15},
		we find for $R\gg1$,
		\begin{align}
			\label{4.17.0}
			\di_\phi\theta(\Phi,\z)^NA(\xi)=&O_{H^k}(\z^0(1+\xi))+O_{H^k}(1+\xi)\\
			\label{4.17.1}
			\abs{\inn{\b W(\Phi,\z)}{L_\z\xi}}\ls& R^{-4}\norm{\xi}_{L^2},\\
			\label{4.17.2}
			\abs{\inn{N_\z(\xi)}{L_\z\xi}}\ls& \norm{\xi}_{H^k}^3,\\
			\label{4.17.3}
			\abs{\inn{f_\al\dot\z^\al}{L_\z\xi}}\ls& R^{-2}(1+\norm{\xi}_{H^k})\norm{\xi}_{H^k}.
		\end{align}
		
		Plugging  \eqref{4.16}-\eqref{4.17.3} back to \eqref{4.22}, we find that
		so long as $\norm{\xi}_{H^k}<1/2$ and $R\gg1$,  there holds
		\begin{equation}
			\label{4.19}
			\begin{aligned}
				\dot \Lambda(t)\le& C_1R^{-3}\norm{\xi}_{H^k}\\&-\norm{L_\z\xi}_{H^{k-4}}^2+(C_2R^{-3}+C_3R^{-1}\norm{\xi}_{H^k})\norm{\xi}_{H^k}^2.
			\end{aligned}
		\end{equation}
		
		By the coercivity of $L_\z$ shown in \lemref{lemA2}, the first term in the second line  of \eqref{4.19} can be bounded by $-\norm{L\xi}_{H^k}^2\le-\beta\norm{\xi}_{H^k}$ for some $\beta>0$ independent
		of $\z$. This, together with the upper bound in \eqref{A1.7}, implies that there exists some $\gamma>0$ independent of $\z$, such that
		\begin{equation}
			\label{4.20}
			\dot \Lambda(t)+\gamma \Lambda(t)\le C_1R^{-3}\norm{\xi}_{H^k}+(C_2R^{-3}+C_3R^{-1}\norm{\xi}_{H^k}-\beta/2)\norm{\xi}_{H^k}^2.
		\end{equation}
		
		Thus, so long as 
		\begin{equation}
			\label{A'}
			C_2R^{-3}+C_3R^{-1}\norm{\xi}_{H^k}-\beta/2<0,
		\end{equation}
		we can drop the last term in \eqref{4.20}  to deduce  that 
		\begin{equation}
			\label{4.23}
			\od{(\Lambda(t)e^{\gamma t})}{t}\ls R^{-3}e^{\g t}.
		\end{equation}
		
		Integrating \eqref{4.23} on $[T,\infty)$
		and using \eqref{A1.7}, we find 
		\begin{equation}
			\label{4.24}
			\norm{\xi_t}_{H^k}^2\ls \Lambda(t)\ls e^{-\g t}\norm{\xi_T}_{H^k}^2+ R^{-3}M(t)\quad (M(t):=\sup_{T\le t'\le t}\norm{\xi_{t'}}_{H^k}),
		\end{equation}
		where $\xi_T$ is the fluctuation at $t=T$. Taking supremum
		of both sides of \eqref{4.24} in $t$, and then dividing by $M(t)$ (which is positive for all $t$ in the nontrivial case), we find
		\begin{equation}
			\label{4.25}
			M(t)\ls e^{-\g t}\norm{\xi_T}+R^{-3}.
		\end{equation}
		This, in particular, implies that the Ansatz \eqref{A'} is satisfied as long as $R\gg1$.
		Iterating this process,
		we get the estimate $M(t)\ls R^{-3}$. Plugging this back to the dynamics \eqref{4.15} and
		\eqref{4.17} for $\z$ and $\xi$ respectively, we find that the velocity of the fluctuation satisfies
		$\norm{\di_t\xi}_{H^k}\ls R^{-4}.$ Thus for sufficiently large $R$, the claim \eqref{claim} follows.  
		
	\end{proof}


	\begin{corollary}
		\label{cor1}
		Let $\z_t\in M'$ be a flow evolving according to \eqref{4B}.
		Then there exist some $T>0$ and a global solution $x_t$ to \eqref{W}, such that \eqref{4A} holds
		on the time interval $T\le t\le T+R$ with this choice of $\z_t$. 
	\end{corollary}
	\begin{proof}
		Let $T>0$ be the first time 
		$u_t$ enters the manifold $U_\delta$ defined in \eqref{2.3} (this $T$ is finite by \thmref{thm2}).
		Consider the flow $x_t,\,t\ge T$ generated under \eqref{W} by $x_*=\Phi(\z_T)$, as in
		\thmref{thm1}. Let $\z_t,\,t\ge T$ be the flow of barycenter as in \eqref{4A}. 
		Then using \eqref{7} and \eqref{4A}, we can estimate
		\begin{equation}
			\label{4.26}
			\begin{aligned}
				\norm{\Phi(\tilde\z_t)-x_t^N}_{H^k}&\le \norm{\Phi(\z_t)-x_t^N}_{H^k}+\norm{\Phi(\z_t)-\Phi(\tilde\z_t)}_{H^k}\\&\ls R^{-2}\abs{\z_t-\tilde\z_t}+R^{-3}\\&\ls R^{-2}\int_{T}^{T+R}\abs{\dot\z_t-\dot{\tilde{\z_t}}}+R^{-3}
				\\&\ls R^{-3}\quad (T\le t\le T+R).
			\end{aligned}
		\end{equation}
		In the last step we use the effective dynamics \eqref{4B}.

	\end{proof}
	
	In the following corollary, we display the $t$-dependence in subscripts.
	\begin{corollary}
		Suppose the embedding $\theta(\Phi(\z),\z)$ parametrizes a surface of Willmore type (i.e.  static solution  to \eqref{W}).
		Then $\theta(\Phi(\z),\z)$ is uniformly stable with small
		area-preserving $H^k$-perturbation if  $\z$ is a strict local
		minimum of the function $G$ defined in \thmref{thm4}.
	\end{corollary}
	\begin{proof}
		Suppose $\z$ is a strict local minimum of $G$.  Then  every flow $\z_t$ 
		starting at some $\z'$ near $\z$ under \eqref{4B} converges to $\z$, i.e.  $\z_t\to \z$ as $t\to \infty$. Now, for this $\z$,
		consider 
		a perturbation $x':=\theta(\Phi(\z)+\xi,\z)$ with $\norm{\xi}_{H^k}\ll1$. By the regularity
		of the nonlinear projection $S$ in \lemref{lem2.1}, this perturbation
		has barycenter $\z'$ close to $\z$. It follows that  the flow of barycenters $\z_t$ associated to the flow of embeddings $x_t$ generated by $x'$ under \eqref{W} satisfies $\abs{\z_t-\z}<\delta$ for any $\delta>0$ and all sufficiently large $t>T$. 
		If we choose $\delta\ll R^{-3}$, then we can conclude from \eqref{4A} that $\norm{x_t^N-\Phi(\z)}_{H^{k-4}}\ls R^{-3}$ for all large $t$.

	\end{proof}
	
	\section*{Acknowledgments}
	
	The Author is supported by Danish National Research Foundation grant
	CPH-GEOTOP-DNRF151. The Author thanks T. K\"orber for the introduction to the subject and helpful remarks.
	
	\section*{Declarations}

	\begin{itemize}
		\item Competing interests: The Author has no conflicts of interest to declare that are relevant to the content of this article.
		\item Data availability: Data sharing is not applicable to this article as no new data were created or analyzed in this study.
	\end{itemize}

	\appendix
	\section{The Abstract Lyapunov-Schmidt map}
	\label{A0}
	Motivated by the synonymous reduction procedure,
	in general we can define Lyapunov-Schmidt map as follows:
	\begin{definition}
		\label{def2.1}
		Let $X\subset Y$ be two Hilbert spaces. 
		Let $U\subset X$ be an open set, and $u_0\in U$.
		Let $P:X\to X$ be an orthogonal projection onto some subspace of $X$, and $\bar{P}:=1-P$ be the complement of $P$.
		
		The Lyapunov-Schmidt map $\Phi=\Phi_{P,u_0}$ 
		is defined on the following set:
		
		$$ \Set{\begin{aligned}
				F:U\subset X\to Y:& \text{$F$ is $C^1$, $\bar{P}L\bar{P}: \bar{P}X\to \bar{P}Y$ has bounded inverse, }\\&\text{where $L:=dF(u)|_{u=u_0}:X\to Y$ is}\\&\text{ the linearized operator of $F$ at $u_0$}.
			\end{aligned}
		}$$
		
		For such $F$, by Implicit Function Theorem, there exists an
		open neighbourhood
		$V\subset PX$ around $v=0$ and a $C^1$ map $w: V\to \b P X$
		such that 
		$\b{P}F(u_0+v+w(v))=0$ for every $v\in V$.
		This map $w$ parametrizes a manifold $E$ in $X$. This manifold $E$
		is a normal perturbation of  $\ran P$. 
		
		Let $w'$ be the Fr\'echet
		derivative of $w$. (This is a map from $V$ to the linear operators on $X$.)
		Then the tangent space at $w=w(v)$ to $E$
		is spanned by $w'(v)(V)$, and can be trivialized as a subspace in $X$,
		with dimension up to $\dim V$. 
		
		Now, the image $\Phi(F)$ of the Lyapunov-Schmidt map
		is given by 
		$$\fullfunction{\Phi(F)}{V\subset PX}{X}{v}{Q_vF(u_0+v+w(v))},$$
		where $w=w(v)$ is as above, and $Q_v$ is the projection onto the tangent space $T_{w(v)}E$.
	\end{definition}
	
	The notion above is first informally conceived in 
	\cite{MR3824945}, in connection with the Feshbach-Schur map introduced in \cites{MR1639713,MR1639709}.
	Here, notice that the  domain $V$ of the map $\Phi(F)$ 
	depends on the spectrum of $L$  only.
	
	Essentially, the map $\Phi$ identifies a reduced space
	which can be either finite-dimensional or otherwise more tractable. 
	The behavior of a map in the domain of $\Phi$ on this reduced space, 
	which is locally isomorphic to 
	$PX$, determines the behavior of
	this map in the vicinity of $0$ in the full space $X$. 
	
	As such, one can view this map $\Phi$ in the context of the
	theory of infinite-dimensional invariant manifolds for semiflows in Banach spaces.
	See \cite{MR1445489, MR1675237,MR2439610} and the references therein.
	
	We will treat the analytical and geometric properties of
	the abstract Lyapunov-Schmidt map elsewhere.

	In the context of this paper, the family of maps involved is $F(\cdot, r,z):H^k\to H^{k-4}$, 
	defined in \defnref{defn2}. These maps are parametrized by 
	the manifold $M'$. 
	The projection $P$ is defined as the orthogonal projection onto the subspace
	$\spn\Set{1,y^1,y^2,y^3}\subset H^k$. The vector $u_0$ is the zero function in $H^k$,
	which, through $\theta$ given in \eqref{3}, corresponds to a large coordinate sphere. 
	\section{Area-rreserving property of \eqref{W}}\label{sec:ap}
	In this section, following \cite{MR4236532}*{Sec. 2}, we show that \eqref{W} is area-preserving.
	
	Let $x=x_t:\Sb\to (M,g)$ be a family of embeddings, 
	and $\Si:=x(\Sb)$.
	Suppose $x$ solves \eqref{W}, namely
	\begin{align}
		\label{ap1}
		\di_tx^N=& -W(x)-\l H(x),\quad W(x):=\Lap H(x)+H(x)(\Ric_M(\nu,\nu)+\lvert\Acirc\rvert^2(x)),\\
		\label{ap2}
		\l=&\l(t)=\frac{1}{\int_{\Si}H^2\,d\mu}\int_{\Si}\del{\abs{\grad H}^2-H^2\Ric_M(\nu,\nu)-H^2\lvert\Acirc\rvert^2}\,d\mu,
	\end{align}
	where $\mu= \mu_\Si^g$ is the  canonical measure on $\Si$, induced by the embedding $x$ and background metric $g$.
	
	If $\l$ is given by \eqref{ap2}, then 
	integration by part on the first term in the r.h.s. of \eqref{ap2} yields
	\begin{equation}
		\label{ap3}
		\l(t)\int_{\Si_t}H^2\,d\mu=-\int_{\Si}WH\,d\mu.
	\end{equation}
	On the other hand, for $\Si=\Si_t=x_t(\Sb)$, the first variation formula
	for the area functional reads
	\begin{equation}\label{ap4}
		\di_t\abs{\Si}=-\int_\Si H\di_tx^N\,d\mu.
	\end{equation}
	Plugging the expression for $\di_tx^N$ from \eqref{ap1} into \eqref{ap4}, we find 
	\begin{align*}
		\di_t\abs{\Si}=& \int_\Si (W+\l H)H\,d\mu\\
		=&0,
	\end{align*}
	where the last step follows from identity \eqref{ap3}.
	This shows \eqref{W} is indeed area-preserving.
	
	\section{Asymptotics of $\W$}\label{sec:A2}
	In this section, we record various expansions related to 
	the  Willmore energy $\W$ at large coordinate spheres.
	Most of the results can be read off from  \cite{eichmair2021large}*{Sec. C},
	\cite{MR2785762}*{Secs. 3,7}.
	
	Fix some $R\gg1$ and $0<\delta\ll1$. 
	In what follows, we restore the parameter $r,z$ and assume   $r\ge R$, $z\in \Rb^3,\,\abs{z}<1-\delta$. 
	
	\begin{proposition}
		\label{propB1}
		For $\Omega$ given in \eqref{2.2.1}, $G$ given in \thmref{thm4}, there hold
		\begin{align}
			\label{B1.1} \Omega(0,r,z)=&4\pi-16\pi r^{-1}+2\pi r^{-2}\del{\frac{10-6\abs{z}^2}{(\abs{z}^2-1)^2}+3\frac{1}{\abs{z}}\log\frac{1+\abs{z}}{1-\abs{z}}}+ O(r^{-3}),\\
			\label{B1.2} G(r,z)=&\Omega(0,r,z)+ O(r^{-5}).
		\end{align}
	\end{proposition}
	\begin{proof}
		\eqref{B1.1} can be found in \cite{eichmair2021large}*{Lem. 42}.
		To derive this expansion, one uses the assumption \eqref{1.3}
		that the scalar curvature $\Sc$ is asymptotically even.
		\eqref{B1.2}
		follows from the expansion $$G(r,z)=\Omega(0,r,z)+\inn{\b W(0,r,z)}{\Phi(r,z)}+ O (\norm{\Phi(r,z)}_{H^k}^2),$$ the expansion for $\b W(0,r,z)$ below, and the
		estimate \eqref{7}. 
	\end{proof}
	
	The following results are used to derive the key estimate \eqref{7}.
	\begin{proposition}
		\label{propB2}
		Let $k\ge 4$. Let $\phi=\phi_{r,z}\in H^k$ be a family of functions with $\norm{\phi}_{H^k}\ll1$, and suppose
		the map $(r,z)\mapsto \phi_{r,z}$ is $C^2$.
		Let  $\b W$ be as in \eqref{2.2}. Then for all $m+\abs{\al}\le2$, there hold
		\begin{align}
			\label{B1.3}
			\b W(\phi,r,z)=&\b W(0,r,z)+ L_{r,z}\phi +N_{r,z}(\phi),\\
			\label{B1.4}
			\norm{\di_r^m\di^\al_z\b W(0,r,z)}_{H^{k-4}}=& O(r^{-(4+m)}),\\
			\label{B1.5}
			\norm{\di_r^m\di^\al_zL_{r,z}\phi}_{H^{k-4}}
			\ls& \norm{\di_r^m\di^\al_z\phi}_{H^k},\\
			\label{B1.6}
			\norm{\di_r^m\di^\al_zN_{r,z}(\phi)}_{H^{k-4}}\ls& \norm{\di_r^m\di^\al_z\phi}_{H^k}^2.
		\end{align}
	\end{proposition}
	\begin{proof}
		The expansion \eqref{B1.3} is valid by the regularity of $\b W$ from $Y^k\to H^{k-4}$.
		\eqref{B1.4} follows from the formula \cite{eichmair2021large}*{Cor. 45}. Since 
		$L_{r,z}$ is linear in $\phi$, 
		\eqref{B1.5} is immediate. For \eqref{B1.6}, we calculate $N_{r,z}(\phi)$ explicitly as
		\begin{equation}
			\label{B1.7}
			\begin{aligned}
				N_{r,z}(\phi)=&\Lap H(\theta(\phi,r,z))+H(\theta(\phi,r,z))\Bigl(\Ric_M(\nu,\nu)+\lvert\Acirc\rvert^2(\theta(\phi,r,z))\Bigr)
				\\
				&-\Lambda^2_{r,z}\phi-\frac{1}{2}H^2_{r,z}\Lambda_{r,z}\phi-\lambda \Lambda_{r,z}\phi\\&-2H_{r,z}g(\Acirc_{r,z},\grad^2\phi)-2H_{r,z} \Ric_M(\nu,\grad \phi)+2 \Acirc_{r,z}(\grad H_{r,z},\grad \phi)\\
				&-V_{r,z}\phi+O(r^{-4}),
			\end{aligned}
		\end{equation}
		where $\lambda$ is the Lagrange  multiplier in \eqref{W}, 
		$\Lambda=-\Lap-(\abs{A}^2+\Ric_M(\nu,\nu))$ is the Jacobi
		operator, $V$ is some smooth 
		function depending on the sphere $\theta(0,r,z)$ and the background
		metric only, and 
		the subscript on $r,z$ indicates  evaluation at the sphere $\theta(0,r,z)$. 
		Formula \eqref{B1.7} follows from \cite{MR2780248}*{Sec. 7.2}. 
		
		Now, inspecting each term in \eqref{B1.7}, one can see from this explicit expression that the nonlinearity  $N_{r,z}$ 
		is a finite sum of the form  $N_{r,z}=N_{r,z}(D^{\le4}\phi,g)$, where $D^{\le 4}\phi$ denotes all derivatives of $\phi$
		up to order $4$. Moreover, each term in this finite sum, together with its derivatives,
		is uniformly bounded by some $C>0$  independent of $r,z$ as a map from $H^k\to H^{k-4}$ . This implies $$\norm{\di_r^m\di^\al_zN_{r,z}(\phi)}_{H^{k-4}}\ls \norm{\di_r^m\di^\al_z\phi}_{H^k}^M$$
		for some $M=M(m,\al)\ge 2$. For $\norm{\phi}_{H^k}\ll1$, this implies \eqref{B1.6}.

	\end{proof}
	
	By perturbing the estimates in \propref{propB2}, we get the following results:
	\begin{proposition}
		\label{propB3}
		Let $k\ge 4$. Let $\phi=\phi_{r,z}\in H^k$ be a family of functions with $\norm{\phi}_{H^k}\ll1$.
		Let  $\b W$ be as in \eqref{2.2}. 
		
		Suppose
		the map $(r,z)\to \phi_{r,z}$ is $C^2$, and satisfies \eqref{7}.
		Then,  for every $\norm{\xi}\ll1$ , there hold
		\begin{align}
			\label{B1.3'}
			\b W(\phi+\xi,r,z)=&\b W(\phi,r,z)+ \tilde L_{r,z}\xi +\tilde N_{r,z}(\xi),\\
			\label{B1.4'}
			\norm{\b W(\phi,r,z)}_{H^{k-4}}=& O(r^{-4}),\\
			\label{B1.5'}
			\norm{\di_r^m\di^\al_z\tilde L_{r,z}\xi}_{H^{k-4}}
			\ls& \norm{\di_r^m\di^\al_z\xi}_{H^k}\quad(m+\abs{\al}\le1),\\
			\label{B1.6'}
			\norm{\tilde N_{r,z}(\xi)}_{H^{k-4}}\ls& \norm{\xi}_{H^k}^2.
		\end{align}
	\end{proposition}
	\begin{proof}
		The point is that, by the regularity of $\Omega(\cdot,r,z)$ on $H^k$ with $k\ge4$,
		the map $\b W(\cdot,r,z)$ together with its derivatives varies
		continuously as $\phi$ varies.
		Thus, if $\phi$ satisfies \eqref{7}, then \eqref{B1.4'}-\eqref{B1.6'} follow from \eqref{B1.4}-\eqref{B1.6} with appropriate choice of $r,\al$. 
	\end{proof}
	
	\section{Spectral properties of $L_{\z}$}
	Let $k\ge 4$.
	Fix $\z=(r,z^j)\in M'\subset \Rb_{>R}\times B_{1}(0)$ for sufficiently large $R$.
	In this section, 
	we prove various uniform estimates for the linearized operator $L_{\z}$
	of $\b{W}(\cdot,\z)$ at $\Phi(\z)$.
	For more details, see \cite{MR2785762}*{Sections 3, 7}.

	Recall  $P:H^k\to  H^k$ is the $L^2$-orthogonal projection onto the span consisting of the constants and the three spherical coordinates.
	$Q=Q_\z: H^k\to H^k$ is the (not necessarily orthogonal) projection onto
	the tangent space at
	$\theta(\Phi(\z),\z)^N$ to $E^N:=\Set{\theta^N:\theta\in E}\subset H^k$,
	where the manifold $E$ is given in \defnref{LSM}.
	
	\begin{lemma}[approximate zero modes]
		\label{lemA1}
		For every $\phi\in H^k$, there holds
		\begin{equation}
			\label{A1.1}
			\norm{L_\z Q\phi}_{H^{k-4}}\ls O(R^{-2}) \norm{\phi}_{H^k}.
		\end{equation}
	\end{lemma}
	\begin{proof}
		Recall that the functions $f_\al$ are calculated in \eqref{2.17}-\eqref{2.18},
		and that the projection $Q$ maps onto $\spn\Set{f_\al}$. Using these facts, 
		we have shown $\norm{Q-P}_{H^k}\ls R^{-2}$  as  in \eqref{3.2.1}. 
		
		Denote by $L^0$ the linearized operator at the coordinate sphere $S_{\z}$ with
		Euclidean background metric, and by $L^g$ the same with the metric $g$ on $M$. 
		These two operators satisfy 
		\begin{equation}
			\label{A1.2}
			\norm{L^0_\z\phi}_{H^{k-4}}=\norm{L^g_\z\phi}_{H^{k-4}}+O(R^{-3})\norm{\phi}_{H^k}.
		\end{equation} 
		See a discussion about this in \remref{rmk1}.
		
		By the definition of $P$ in \defnref{defn2}, $L^0$
		vanishes on $\ran P$. It follows that 
		\begin{equation}
			\label{A1.3}
			\norm{L^0_\z Q\phi}_{H^k}\le\norm{L^0_\z }_{H^k\to H^{k-4}}\norm{(P-Q)\phi}_{H^k}=O_{H^k}(R^{-2})\norm{\phi}_{H^k}.
		\end{equation}
		Since the functional $\W$ is $C^2$ on $X^k$, we have by \eqref{7} that
		\begin{equation}
			\label{A1.4}
			\norm{L_\z-L^g_\z}_{H^k\to H^{k-4}}\ls \norm{\z^\al\di_{\z^\al}\Phi}_{H^k}=O(R^{-2}).
		\end{equation}
		The claim follows by combining \eqref{A1.2}-\eqref{A1.3} with \eqref{A1.4}.
	\end{proof}
	
	\begin{lemma}[coercivity]\label{lemA2}
		Let $k\ge4,\,R\gg1$. 
		There exist $\al,\,\beta>0$ depending on $k,\,R$ and  the background metric $g$ only,
		such that
		\begin{equation}
			\label{A1.7}
			\al\norm{\xi}_{H^k}^2\le \inn{\xi}{L_\z\xi}\le  \beta\norm{\xi}_{H^k}^2\quad (\xi\in \ker Q).
		\end{equation}
	\end{lemma}
	\begin{proof}
		The upper bound follows from the fact that $\b W$ is $C^2$ and the relation
		$L_\z=\di_\phi \b W(\Phi(\z),\z)$.
		Thus it suffices to find the lower bound.
		
		In \cite{MR2785762}*{Thm. 10}, it is shown that 
		\begin{equation}
			\label{A1.8}
			\inn{\xi}{L_\z\xi}\ge \al\norm{\xi}_{L^2}^2\quad (\xi\in \ker P),
		\end{equation}
		where $\al$ depends only on the lower bound $R$ and the ambient metric $g$.
		The claim now is 
		\begin{equation}
			\label{A1.9}
			\inn{\xi}{L_\z\xi}\ge \frac{\al}{4}\norm{\xi}_{L^2}^2\quad (\xi\in \ker Q),
		\end{equation}
		provided $R$ is sufficiently large. Notice that here we have one caveat, namely that 
		$Q$ is not necessarily orthogonal.
		
		Put $\b{Q}=1- Q$. Then we can rewrite \eqref{A1.9} as 
		\begin{equation}
			\label{A1.10}
			\inn{L_\z\xi}{\xi}=\inn{L_\z Q\xi}{\xi}+\inn{L_\z\b Q\xi}{\xi}.
		\end{equation}
		The first term is $O(R^{-2})\norm{\xi}_{H^k}^2$ by the approximate
		zero mode property \eqref{A1.1}. The second term further splits
		as $\inn{L_\z\b Q\xi}{\b Q\xi }+\inn{L_\z\b Q\xi}{Q\xi}=\inn{L_\z\b Q\xi}{\b Q\xi }+\inn{\b Q\xi}{LQ\xi}=\inn{L\b Q\xi}{\b Q\xi }+O(R^{-2})\norm{\xi}_{H^k}^2$,
		again by \eqref{A1.1}. Thus it suffices to find a lower bound of the quadratic form $\inn{L\b Q\xi}{\b Q\xi }\ge (\al/2) \norm{\xi}_{H^k}$.
		
		Consider the operator $\b QL \b Q$. By the uniform closeness \eqref{3.2.1}, we have 
		$\b QL \b Q=\b PL\b P+O_{H^k\to H^{k-4}}(R^{-2})$. This and $k\ge 4$ imply that 
		for sufficiently large $R$, we have  $\norm{\b QL \b Q}_{H^k\to L^2}\ge \al/2$ and therefore \eqref{A1.9}.
		
		Thus \eqref{A1.7} follows once we can find a uniform estimate $\norm{\xi}_{H^k}\ls \norm{\xi}_{L^2}. $
		This can be done using the fact that 
		$\xi$ solves \eqref{4.17}, which we can rewrite as a fourth order
		elliptic equation
		$$L_\z\xi=-\b W(\Phi,\z)-N_\z(\xi)+O(R^{-2}).$$
		By the regularity theory for elliptic operator of higher-order, this implies 
		that there is some $C>0$ depending on $R$, the background metric $g$,
		and the Sobolev order $k$ only, such that $\norm{\xi}_{H^k}\le C \norm{\xi}_{L^2}. $
		
	\end{proof}
	
	%
	%
	
	\bibliography{bibfile}
\end{document}